\documentclass[11pt]{article}
\pdfoutput=1 % for arXiv
\usepackage[height=8in]{geometry}
\usepackage{amsmath, amssymb, amsthm}
\usepackage{graphicx, color}
\usepackage{array, multirow}
\usepackage[font=small,labelfont=bf]{caption}

\usepackage[sort]{natbib}

\setcitestyle{numbers,square,comma}
\usepackage{hyperref}
\hypersetup{colorlinks=true,
			urlcolor=blue,
			linkcolor=blue,
			citecolor=blue,
			bookmarksdepth=paragraph,
			pdftitle={Convex relaxation of variational problems using sub-osculators and polynomial optimization (arXiv version)},
			pdfauthor={Alexandr Chernyavskiy, Jason J.\ Bramburger, Giovanni Fantuzzi, David Goluskin},}
\usepackage[nameinlink]{cleveref}
\crefname{equation}{}{}
\crefname{section}{section}{sections}
\crefname{figure}{figure}{figures}
\crefname{table}{table}{tables}
\crefname{example}{example}{examples}
\Crefname{section}{Section}{Sections}
\Crefname{figure}{Figure}{Figures}
\Crefname{table}{Table}{Tables}
\Crefname{definition}{Definition}{Definitions}
\Crefname{theorem}{Theorem}{Theorems}
\Crefname{remark}{Remark}{Remarks}
\Crefname{example}{Example}{Examples}
\numberwithin{equation}{section}

\theoremstyle{definition}

\usepackage{titlesec}
\usepackage{hyperref}
\usepackage{amssymb}
\usepackage{amsthm}
\usepackage{ulem}
\usepackage{amsmath}
\usepackage{amsfonts}
\usepackage{epstopdf}
\usepackage{mathtools,nccmath,relsize}
\usepackage{xcolor}
\usepackage{bbm}
\usepackage{esint}
\usepackage{enumerate}
\usepackage{calrsfs}
\mathtoolsset{showonlyrefs}
\newtheorem{Thm}{Theorem}[section]
\newtheorem{Prop}[Thm]{Proposition}
\newtheorem{Lmm}[Thm]{Lemma}
\newtheorem{Cor}[Thm]{Corollary}
\newtheorem{Rmk}[Thm]{Remark}
\theoremstyle{plain} % just in case the style had changed
\newcommand{\thisTheoremname}{}
\newtheorem{genericthm}[Thm]{\thisTheoremname}

\allowdisplaybreaks
\newcommand{\ZT}{\mathcal{Z}}
\newcommand{\HM}{\hat{\mathcal{H}}^{\partial M}}

\newcommand{\IV}{\mathcal{I}_{\ZT}}
\newcommand{\EV}{\mathcal{E}_{\ZT}}
\newcommand{\Iz}{\mathcal{I}_{\ZT=0}}
\newcommand{\Ez}{\mathcal{E}_{\ZT=0}}

\newcommand{\RRR}{\mathbb{R}^3}
\newcommand{\R}{\mathbb{R}}
\newcommand{\Hone}{\mathcal{H}^1(\mathbb{R}^3)}
\newcommand{\Honehat}{\hat{\mathcal{H}}^1(\mathbb{R}^3)}
\newcommand{\Lone}{\mathcal{L}^1(\mathbb{R}^3)}
\newcommand{\Ltwo}{\mathcal{L}^2(\mathbb{R}^3)}
\newcommand{\Linfty}{\mathcal{L}^\infty(\mathbb{R}^3)}

\newcommand{\N}{\mathbb{N}}
\newcommand{\supp}{\text{supp}\,{}}
 \renewcommand{\d}{\,\text{d}}

\newcommand{\ntends}{\xrightarrow[n \rightarrow \infty]{}}

\title{An Ohta-Kawasaki Model set on the space}
\author{Lorena Aguirre Salazar$^1$, Xin Yang Lu$^2$, Jun-cheng Wei$^3$}
\date{\small $^1$Department of Mathematical Sciences, Lakehead University, One Georgian Dr., Barrie, ON, L4M 3X9, Canada. Email: \texttt{lorena.aguirresalazar@lakeheadu.ca}\\
$^2$Department of Mathematical Sciences, Lakehead University,
955 Oliver Rd., Thunder Bay, ON, P7B 5E1, Canada. Email: \texttt{xlu8@lakeheadu.ca}\\
$^3$Department of Mathematics, University of British Columbia, Vancouver, BC, V6T 1Z2, Canada. Email: \texttt{jcwei@math.ubc.ca}
}

\begin{document}

\maketitle

\begin{abstract}
We examine a non-local diffuse interface energy with Coulomb repulsion in three dimensions inspired by the Thomas-Fermi-Dirac-von Weizs\"{a}cker, and the Ohta-Kawasaki models. We consider the corresponding mass-constrained variational problem and show the existence of minimizers for small masses, and the absence of minimizers for large masses.
\end{abstract}

%%%%%%%%%%%%%%%%%
%% MAIN TEXT
\section{Introduction}

We frequently encounter variational problems featuring competing terms in studies related to energy-driven pattern formation. The Thomas-Fermi-Dirac-von Weizs\"{a}cker (henceforth TFDW) and the Ohta-Kawasaki models stand as representative functionals that have received increasing attention~ \cite{FrankNamVanDenBosch,aguirre2020mass,aguirre2021convergence,choksi2001scaling,choksi2011small,choksi2009phase,goldman2013gamma,goldman2014gamma,lu2014nonexistence,muratov2010droplet,spadaro2009uniform}.\newline %check more recent works, add TFDW ones

The TFDW theory~\cite{le2005atoms,Lieb1981} is a density functional theory that is used to approximate the many-body Schr\"{o}dinger theory. Mathematically, the TFDW theory is defined (up to rescaling) by the energy functional
\begin{align}
\int_{\RRR}\left(\vert \nabla u\vert ^2+c_1|u\vert^{\frac{10}{3}}-c_2|u\vert^{\frac{8}{3}}-\ZT \frac{u^2}{\vert x\vert}\right)\d x+\frac{1}{2}\int_{\RRR}\int_{\RRR}\frac{ u^2({x}) u^2({y}) }{\vert {x}-{y}\vert }\d x\d {y},
\end{align}
with $c_1,c_2>0$, and $\ZT\geq0$, on the class of functions $u\in H^1(\RRR)$ with a prescribed $\mathcal{L}^2$-norm. Each such $u$ represents an electron density function with mass given by its $\mathcal{L}^2$-norm. The energy is to be thought of as the energy of a system of a fixed number of electrons interacting with a nucleus of charge $\ZT$ fixed at the origin. Finding the infimum of the energy makes sense because Chemical and Physical systems are usually found in their most stable state, and that corresponds to the lowest energy possible. The infimum corresponds to the ground state energy, an optimal $u$ corresponds to a state or electronic configuration of optimal energy, and such $u$ sheds light on properties of an atom.\newline

The first density functional theory was the Thomas-Fermi (henceforth TF) theory~\cite{fermi1927metodo,thomas1927calculation}, a theory that captures the leading-order behavior of the ground state energy of atoms in the large $\ZT$ limit. But negative ions were absent in this theory~\cite{LiebSimonsTF}. Then, a leading-order correction was incorporated by adding the von Weizs\"{a}cker gradient term~\cite{weizsacker1935theorie} to the energy functional. In the Thomas–Fermi–von Weizsäcker (henceforth TFW) theory, negative ions do exist while arbitrarily negative ions do not~\cite{BenguriaBrezisLieb}. Finally, a second-order correction to the TF theory was obtained by adding Dirac's term~\cite{dirac1930note} to the energy functional, the term with power $8/3$.\newline %The two corrections make the accuracy of the TFDW theory comparable to that of the Hartre Fock theory~\cite{bach1992error}, an accurate density matrix theory, when $\ZT$ is large.\newline

Regarding existence of minimizers of the TFDW model, there exist $\epsilon_1,\epsilon_2>0$ such that there exists a minimizer for masses less than $\ZT+\epsilon_1$~\cite{lions1987solutions,le1993thomas}, and there are no minimizers for masses larger than $\ZT+\epsilon_2$.\newline

On the other hand, the Ohta-Kawasaki model was originally introduced in ~\cite{OhtaKawasaki} in the context of microphase separation in diblock copolymer melts. A diblock copolymer molecule is a linear chain consisting of two subchains made of two different monomers joined covalently to each other. Microphase separation occurs as monomers of the same type attract while monomers of opposite type repel.\newline

The Ohta-Kawasaki model is defined (up to rescaling) by the energy functional
\begin{equation}\int_\Omega\left[\frac{\epsilon}{2}\vert \nabla u \vert^2+\frac{1}{4\epsilon}(1-u^2)^2\right]\d x+\frac12\int_\Omega\int_\Omega G(x,y)[u(x)-m][u(y)-m]\d x\d y,\end{equation}
where $\Omega\subset\mathbb{R}^3$ is a fixed open set, the domain occupied by the material, $\epsilon>0$ is a parameter that is proportional to the thickness of the transition regions between the two monomers, $G$ is the Neumann Green's function of the Laplacian, $u\in H^1(\Omega)$ is the scalar order parameter, and $m:=\fint_\Omega u\d x\in(-1,1)$, the background charge density, is prescribed. The function $u$ is the difference between the averaged densities of the two monomers, so $u$ takes values between $-1$ and $1$ and $u=\pm1$ when there is a concentration of a single monomer.\newline

It is worth noting that the Ohta-Kawasaki model extends its relevance to a broad spectrum of other physical systems ~\cite{chen1993dynamics,de1979effect,nagaev1995phase,nyrkova1994microdomain,glotzer1995reaction,lundqvist1983density,emery1993frustrated,lattimer1985physical,maruyama2005nuclear,muratov2002theory}; it corresponds to a diffuse interface version of the Liquid Drop model~\cite{ChoksiMuratovTopaloglu} in the sense of $\Gamma$-convergence, and to a Cahn-Hilliard model~\cite{Sternberg} with a non-local term. \newline

Regarding existence of minimizers, it is possible to use the direct method of the Calculus of Variations to show that a minimizer will always exist for all choices of the parameters, and $\Omega$ smooth and bounded.\newline

In this work, we study the existence of minimizers of a non-local diffuse interface energy posed on the space, inspired by the TFDW and the Ohta-Kawasaki models. More precisely, we consider the problem
\begin{align}\IV(M):=\inf\left\{\EV(u);u\in\HM\right\},\end{align}
where the energy functional $\EV$ is defined as
\begin{align}\EV(u)&:=\int_{\RRR}\left[\frac{\vert\nabla u\vert^2}{2}+ \frac{1}{2}u^2(1-u)^2-Vu\right]\d x+\frac{1}{2}\int_{\RRR}\int_{\RRR}\frac{u({x})u({y})}{\vert x-y\vert}\d x\d y.\end{align}
with $V:\RRR\to\R^+$ given by
\begin{align}\label{VTFDW}V( x ):=\frac{\ZT}{\vert x\vert},\end{align}
and
\begin{align}\HM:=\left\{u\in\Honehat:u\geq0\text{ a.e. in }\RRR\text{ with }\vert\vert u\vert\vert_{\Lone}=M\right\},\end{align}
where $\Honehat=\overline{\mathcal{C}_0^{\infty}(\RRR)}$ with respect to the norm $\vert\vert \cdot\vert\vert_{\Lone}+\vert\vert \nabla\cdot\vert\vert_{\mathcal{L}^2(\RRR)}$.

%Our choice for $V$ ensures $\IV$ is finite, $\EV$ is coercive in $\Honehat$ on the constraint set, and the mapping
%\begin{align}u\in\Honehat\mapsto\int_{\RRR}Vu\d x\end{align}
%is weakly lower semicontinuous in $\Honehat$.\newline

Our main result is the following one:

\begin{Thm}\label{Th1} There exist constants $0<\ZT\leq M_1\leq M_2<\infty$ such that:
\begin{enumerate}[(i)]
\item If $M\leq M_1$, then there is a minimizer.
\item If $M\geq M_2$, then there are no minimizers.
\end{enumerate}
\end{Thm}

\begin{Rmk}While we expect $M_1=M_2=\ZT$, it remains open to prove or disprove this.\end{Rmk}

The paper is organized as follows. In Section 2 we describe some basic properties of the energy functional and its minimizers. In Section 3 we prove part (a) of Theorem \ref{Th1}. Finally, in Section 4 we prove part (b) of Theorem \ref{Th1}. Our overall strategy is to adapt some approaches and techniques developed in~\cite{BenguriaBrezisLieb} for establishing the existence of minimizers of the TFW energy for sufficiently small masses, and in~\cite{FrankNamVanDenBosch} for the nonexistence of minimizers of the TFDW energy for large masses. There are important changes we outline as proofs unfold.\newline

In what follows, we use the notation
$$\mathcal{D}(f,g):=\frac{1}{2}\int_{\RRR}\int_{\RRR}\frac{f({x})g({y})}{\vert x-y\vert}\d x\d y.$$
We will refer to $\mathcal{D}(u,u)=:\mathcal{D}(u)$ as the {\em Coulomb repulsion term}, $\frac12\int_{\mathbb{R}^3} u^2(1-u)^2 \d x$ as the {\em double well term}, and
$-\int_{\mathbb{R}^3}Vu \d x$
as the {\em attraction term}. Denote by $B_R$ the ball centered around the origin with radius $R$. Finally, $C$ will denote some (positive) universal constant that might change from line to line.

\section{General Estimates}

The goal of this section is to establish some properties of the energy functional and its minimizers.\newline

Our first result tells us that the condition $u\geq0$ plays no role as long as $\vert\vert u\vert\vert_{\Lone}$ is not too large.

\begin{Lmm}\label{OnObstacleForSmallMass}Let $u\in\hat{\mathcal{H}}^1(\RRR)$. If $\vert\vert u\vert\vert_{\Lone}\leq\ZT$, then $\EV(\vert u\vert)\leq \EV(u)$.\end{Lmm}
\begin{proof}We have
\begin{align}2\mathcal{D}(u_-,u_+)-\int_{\RRR}Vu_-\d x&=2\mathcal{D}(u_-,u_+)-\ZT\int_{\RRR}\frac{u_-(x)}{\vert x\vert}\d x\\
 &\leq2\mathcal{D}(\overline{u_-},\overline{u_+})-\ZT\int_{\RRR}\frac{\overline{u_-}({x})}{\vert x\vert}\d x\\
 &=\int_{\RRR}\left(\overline{u^+}\ast\vert\cdot\vert^{-1}-\frac{\ZT}{\vert x\vert}\right)\overline{u^-}({x})\d x,\end{align}
 where the line over functions corresponds to their spherical average. Moreover, by equation (35) in~\cite{LiebSimonsTF}, we have
 \begin{equation}\overline{u^+}\ast\vert\cdot\vert^{-1}(x)\leq\frac{\vert\vert \overline{u^+} \vert\vert_{\Lone}}{\vert x\vert}, \quad \vert x\vert>0.\end{equation}
 Consequently,
 \begin{align}2\mathcal{D}(u_-,u_+)-\int_{\RRR}Vu_-\d x
 &\leq\int_{\RRR}\left(\vert\vert \overline{u^+}\vert\vert_{\Lone}-\ZT\right)\frac{\overline{u^-}({x})}{\vert {x}\vert}\d x\leq0.\end{align}

As a result,
 \begin{align}\mathcal{D}(u)-\int_{\RRR}Vu&=\mathcal{D}(u_+,u_+)-2\mathcal{D}(u_-,u_+)+\mathcal{D}(u_-,u_-)-\int_{\RRR}Vu_+\d x+\int_{\RRR}Vu_-\d x\\
 &\geq \mathcal{D}(u_+,u_+)+2\mathcal{D}(u_-,u_+)+\mathcal{D}(u_-,u_-)-\int_{\RRR}Vu_+\d x-\int_{\RRR}Vu_-\d x\\
 &=\mathcal{D}(\vert u\vert,\vert u\vert)-\int_{\RRR}V\vert u\vert \d x.\end{align}
On the other hand, all other terms in $\EV(u)$ do not increase if we replace $u$ by $\vert u\vert$. Consequently, the result follows.

\end{proof}

\begin{Cor} If $M\leq \ZT$, then
\begin{align}\label{OKmodelNoAbsVlNnlcNoObstcl}\IV(M)=\inf\left\{\EV(u);u\in\hat{\mathcal{H}}^1(\RRR),\vert\vert u\vert\vert_{\Lone}=M\right\}.\end{align}
\end{Cor}
\begin{proof}This is an immediate consequence of the previous Lemma.\end{proof}

The next result concerns continuity and coercivity of the energy functional.

\begin{Lmm}\label{CrcvWkLscOK} The energy functional $\EV$ is continuous over $\Honehat$, and the following hold for all $u\in\Honehat$:
\begin{align}\label{Crcv1}\EV(u)+C\ZT^2&\geq\int_{\RRR}\left[\frac{1}{4}\vert\nabla u\vert^2+ \frac{1}{2}u^2(1-u)^2\right]\d x+\frac{1}{2}\mathcal{D}(u),\end{align}
\begin{align}\label{Crcv2}\EV(u)+C\left[\ZT^2+\left(\int_{\RRR}\vert\nabla u \vert^2\d x\right)^3\right]&\geq\int_{\RRR}\left[\frac{1}{4}\vert\nabla u\vert^2+ \frac{1}{4}(u^2+u^4)\right]\d x+\frac{1}{2}\mathcal{D}(u).\end{align}
\end{Lmm}
\begin{proof} The continuity of $\EV$ is standard.\newline
The proof of \eqref{Crcv1} is similar to that of Lemma 2 in~\cite{BenguriaBrezisLieb} for the TFW energy functional. Indeed, by $-\Delta \left(u\star\vert\cdot\vert^{-1}\right)=4\pi u$ and Sobolev's inequality, we have
\begin{align}\left\vert\left\vert u\star\vert\cdot\vert^{-1}\right\vert\right\vert_{\mathcal{L}^6(\RRR)}^2&\leq C\left\vert\left\vert\nabla \left(u\star\vert\cdot\vert^{-1}\right)\right\vert\right\vert^2_{\Ltwo}\\
&=C\mathcal{D}(u).\end{align}
Now, pick any smooth function $\eta:\RRR\to[0,1]$ for which $\mathbbm{1}_{B_1( {0})}\eta\equiv1$ and $\mathbbm{1}_{\RRR\setminus B_2( {0})}\eta\equiv0$, and define the pair of functions $V_1,V_2:\RRR\to\R$ by
\begin{align}V_1({x}):=V\eta\textit{ and }V_2({x}):=V(1-\eta).\end{align}
% the sobolev constant is \frac{1}{3}\left(\frac{2}{\pi}\right)^{\frac{4}{3}}
Then, by  $-\Delta \left(u\star\vert\cdot\vert^{-1}\right)=4\pi u$, H\"{o}lder's inequality, Sobolev's inequality, and Young's inequality, we have
\begin{align}\int_{\RRR}Vu&=\int_{\RRR}V_1u\d x+\int_{\RRR}V_2u\d x\\
&=\int_{\RRR}V_1u\d x+\frac{1}{4\pi}\int_{\RRR}(-\Delta V_2)u\star\vert\cdot\vert^{-1}\d x \\
&\leq C\ZT\left(\vert\vert u\vert\vert_{\mathcal{L}^6(\RRR)}+\left\vert\left\vert u\star\vert\cdot\vert^{-1}\right\vert\right\vert_{\mathcal{L}^6(\RRR)}\right)\\
&\leq C\ZT[ \vert\vert \nabla u\vert\vert_{\mathcal{L}^2(\RRR)}+ \sqrt{\mathcal{D}(u)}]\\
&\leq C\ZT^2+\frac{1}{4}\int_{\mathbb{R}^3}\vert \nabla u\vert^2+\frac{1}{2}\mathcal{D}(u).\end{align}
Equation \eqref{Crcv1} then follows.\newline
Next, to establish \eqref{Crcv2} we use \eqref{Crcv1} along with the following, which is established using basic properties of the distribution function and Sobolev's inequality:
\begin{align}\int_{\RRR}u_+^3\d x&\leq\frac{1}{4}\int_{\RRR\cap\{u_+\leq 1/4\}}u_+^2\d x+\int_{\RRR\cap\{1/4\leq u_+\leq 4\}}u_+^3\d x+\frac{1}{4}\int_{\RRR\cap\{1/4\leq u\}}u_+^4\d x\\
&\leq\frac{1}{4}\int_{\RRR}u_+^2\d x+C\vert\{1/4\leq u_+\leq 4\}\vert+\frac{1}{4}\int_{\RRR}u_+^4\d x\\
&\leq \frac{1}{4}\int_{\RRR}u_+^2\d x+C\int_{\RRR}u_+^6\d x+\frac{1}{4}\int_{\RRR}u_+^4\d x\\
&\leq \frac{1}{4}\int_{\RRR}u_+^2\d x+C\left(\int_{\RRR}\vert\nabla u_+ \vert^2\d x\right)^3+\frac{1}{4}\int_{\RRR}u_+^4\d x\\
&\leq \frac{1}{4}\int_{\RRR}u^2\d x+C\left(\int_{\RRR}\vert\nabla u \vert^2\d x\right)^3+\frac{1}{4}\int_{\RRR}u^4\d x\end{align}
\end{proof}

Using the previous Lemma, we obtain boundedness of minimizing sequences in $\Honehat$.

\begin{Cor}\label{MnSqBddOK}Let $\{u_n\}_{n\in\mathbb{N}}$ be a minimizing sequence for $\IV(M)$ . Then, $\{u_n\}_{n\in\mathbb{N}}$ is bounded in $\Honehat$.\end{Cor}
\begin{proof}This is an immediate consequence of the Lemma above.\end{proof}

In the case the is no background potential, we can say that the corresponding infimum is the zero function.

\begin{Lmm}\label{InftyFunctionalIsZero}$\Iz\equiv0$.\end{Lmm}
\begin{proof}This follows immediately from $\Ez(u)\geq0$ and
\begin{equation}\Ez(\sigma^{3}u(\sigma\cdot))\xrightarrow[\sigma \rightarrow 0^+]{}0,\quad u\in\Honehat.\end{equation}\end{proof}

Now, we outline some properties of $\IV(m)$.

\begin{Lmm}\label{PrprtsIOK} The following hold:
\begin{enumerate}[(a)]
\item $m\in[0,\infty)\mapsto\IV(m)$ is continuous, nonincreasing, negative (except $\IV(0)=0$), and bounded below.
\item For each $M>0$, there exists $0<m\leq M$ such that $\IV(M)=\IV(m)$ and $\IV(m)$ is attained.
\end{enumerate}
\end{Lmm}
\begin{proof}
The continuity of $m\in[0,\infty)\mapsto\IV(m)$ follows from a standard argument based on the variational principle and appropriate trial states.\newline
Now, let us take $0<m'<m$ and show that $\IV(m)\leq\IV(m')$. Pick two smooth functions $u_1$ and $u_2$ with compact supports for which $\vert\vert u_1\vert\vert_{\Lone}=m'$ and $\vert\vert u_2\vert\vert_{\Lone}=m-m'$. Then, for any vector $x_0\in\RRR$ we have
\begin{equation}
\IV(m)\leq\lim_{n\to\infty}\EV(u_1(\cdot)+u_2(\cdot+nx_0)))=\EV(u_1)+\Ez(u_2).
\end{equation}
Then, we optimize the right-hand side of the equation above over all $u_1$ and $u_2$ and use Lemma \ref{InftyFunctionalIsZero} to conclude that $\IV(m)\leq\IV(m')+\Iz(m-m')=\IV(m')$.\newline
Negativity of $\IV(m)$ for $m>0$ follows from the nonincreasingness of $m\in[0,\infty)\mapsto\IV(m)$ and
\begin{equation}\EV(\sigma u)=-\sigma\int_{\RRR}Vu\d x+\sigma^2\left[\int_{\RRR}\left[\frac{\vert\nabla u\vert^2}{2}+\frac{1}{2}(u^2-2\sigma u^3+\sigma^2u^4)\right]\d x+\mathcal{D}(u)\right]<0,\end{equation}
for $0<\sigma\ll1$. In turn, boundedness from below of $\IV(m)$ is a direct consequence of equation \eqref{Crcv1}.\newline
Next, let $\{u_n\}_{n\in\mathbb{N}}$ be a minimizing sequence for $\IV (M)$. By Lemma \ref{MnSqBddOK}, this sequence is bounded in $\Hone$, hence, up to a subsequence, $u_n\rightharpoonup u_m$ in $\Hone$ and $u_n\to u_m$ almost everywhere in $\RRR$, for some $u_m\in\Hone$. Then, $m:=\vert\vert u_m\vert\vert_{\Lone}\leq M$, and since the energy functional is weakly lower semicontinuous in $\Hone$ and $\IV$ is nonincreasing,
\begin{align}\IV (m)\leq \EV(u_m)\leq\liminf_{n\to\infty}\EV(u_n)=\IV(M)\leq \IV (m).\end{align}
As a result, $\IV(M)=\IV(m)$, where $\IV(m)$ is attained at $u_m$. The reason why $m>0$ is that $\IV(M)<0$.
\end{proof}

The following Proposition is the last ingredient we need to prove part (a) of Theorem \ref{Th1}.

\begin{Prop}\label{NoSlnsZrEig}(Analogue of~\cite[Lemma 12]{BenguriaBrezisLieb}) If $u\in\Honehat$ satisfies
\begin{equation}\label{closeToELZTOK}-\Delta u+u-3u^2+2u^3-V+\vert u\vert\star\vert\cdot\vert^{-1}\geq0,\end{equation}
then $M:=\vert\vert u\vert\vert_{\Lone}\geq\ZT$.
\end{Prop}
\begin{proof}%We follow ideas by R. Benguria, H. Brezis, and E.H. Lieb in~\cite{BenguriaBrezisLieb}. \newline

Let us pick any smooth radial nontrivial function $\xi:\RRR\to[0,1]$ satisfying
\begin{equation}\supp\xi\subset B_2( {0})\setminus B_1( {0}),\end{equation}
and define the sequence of functions $\{\xi_n\}_{n\in\mathbb{N}}:=\{\xi(n^{-1}{x})\}_{n\in\mathbb{N}}$ defined over $\mathbb{R}^3$ and so that
\begin{align}supp\ \xi_n\subset B_{2n}\setminus B_n, \quad n\in\mathbb{N}.\end{align}
We multiply both sides of inequality \eqref{closeToELZTOK} by $\xi_n$ to obtain the family of inequalities
\begin{align}\label{EqnTmsXi}-\Delta u \xi_n+ (u-3 u ^2+2 u ^3)\xi_n\geq(V-\vert u\vert\star\vert\cdot\vert^{-1})\xi_n,\quad n\in\mathbb{N}.\end{align}
On the other hand, we apply H\"{o}lder's inequality to estimate terms on the left-hand side of \eqref{EqnTmsXi} in terms of $n$ as follows:
\begin{align}\left\vert\int_{\RRR}(-\Delta u )\xi_n \d x\right\vert&=\left\vert\int_{B_{2n}( {0})\setminus B_{n}( {0})}(-\Delta u )\xi_n \d x\right\vert\\
&=\left\vert\int_{B_{2n}( {0})\setminus B_{n}( {0})}\nabla u \cdot\nabla\xi_n\d x\right\vert\\
&\leq \vert\vert\nabla u\vert\vert_{\mathcal{L}^2(B_{2n}( {0})\setminus B_{n}( {0}))}\vert\vert\nabla\xi_n\vert\vert_{\mathcal{L}^2(B_{2n}( {0})\setminus B_{n}( {0}))} \\
&=\sqrt{n}\vert\vert\nabla u\vert\vert_{\mathcal{L}^2(B_{2n}( {0})\setminus B_{n}( {0}))}\vert\vert\nabla\xi\vert\vert_{\mathcal{L}^2(B_{2n}( {0})\setminus B_{n}( {0}))}\\
&=\epsilon_n^1\sqrt{n},\end{align}
with $\epsilon_n^1\ntends0$,
\begin{align}\left\vert\int_{\RRR} u ^r\xi_n \d x\right \vert &=\left\vert\int_{B_{2n}( {0})\setminus B_{n}( {0})} u ^r\xi_n\d x\right\vert\\
&\leq\int_{B_{2n}( {0})\setminus B_{n}( {0})}\vert u\vert^r\d x\\
&\leq \vert\vert u^{r}\vert\vert_{\mathcal{L}^2(B_{2n}( {0})\setminus B_{n}( {0}))}\vert\vert 1\vert\vert_{\mathcal{L}^2(\RRR)(B_{2n}( {0})\setminus B_{n}( {0}))}\\
&=\epsilon_n^2n^{\frac{3}{2}},\quad r=1,2,3,\end{align}
with $\epsilon_n^2\ntends0$.

As for the right-hand side of equation \eqref{EqnTmsXi}, we note that
\begin{align}\label{Vtrm}\int_{\RRR}(V-\vert u\vert\star\vert\cdot\vert^{-1})\xi_n\d x&=\int_{\RRR}(\overline{V}-\overline{\vert u\vert\star\vert\cdot\vert^{-1}})\xi_n\d x\\
&=\int_{\RRR}\left(\frac{\ZT}{\vert x\vert }-\overline{\vert u\vert}\star\vert\cdot\vert^{-1}\right)\xi_n({x})\d x\end{align}
where the line over functions corresponds to their spherical average. Moreover, by equation (35) in \cite{LiebSimonsTF}, we have that
\begin{align}\overline{\vert u\vert}\star\vert\cdot\vert^{-1}(x)\leq\frac{M}{\vert x\vert},\quad \vert x\vert>0.\end{align}
As a result, the following holds for $n$ sufficiently large
\begin{align}(\ZT-M)\int_{\RRR}\frac{\xi_n({x})}{\vert{x}\vert}\d x\leq \int_{\RRR}(V-\vert u\vert\star\vert\cdot\vert^{-1})\xi_n\d x.\end{align}
Besides, we can compute for each $n\in\N$
\begin{align}\int_{\RRR}\frac{\xi_n({x})}{\vert{x}\vert}\d x=n^{-1}\int_{\RRR}\frac{\xi(n^{-1}{x})}{\vert n^{-1}{x}\vert}\d x=n^2\int_{\RRR}\frac{\xi(x)}{\vert x\vert}\d x.\end{align}
Consequently, \eqref{EqnTmsXi} holds only if $M\geq\ZT$. This concludes the proof. \end{proof}

Next, we relate the $\mathcal{L}^2$ and the $\mathcal{L}^1$-norms of a minimizer. We will need the following estimates as part of the proof of part (b) of Theorem \ref{Th1}.

\begin{Lmm}
	\label{relating L2 and L1 norms - lemma}
	Assume there exists a minimizer $u\in \Honehat$ with $\|u\|_{ \Lone }=M$. Then it holds
	\begin{align}
		\|u\|_{ \Ltwo }^2&\le 2(\ZT+2)M+ 8\pi \ZT^2,		\label{relating L2 and L1 norms - equation}\\
		\mathcal{D}(u) & \le 2(\ZT+1)M+ 8\pi \ZT^2.
				\label{relating D and L2 norms - equation}
	\end{align}
\end{Lmm}

\begin{proof}
Note that
\begin{align}
\int_{ \mathbb{R}^3 } u^2 \d x & =
\int_{ \mathbb{R}^3 \cap \{u\ge 2\} } u^2 \d x
+\int_{ \mathbb{R}^3 \cap \{u< 2\} } u^2 \d x\notag\\
&\le
\int_{ \mathbb{R}^3 \cap \{u\ge 2\} } u^2(u-1)^2 \d x
+2\int_{ \mathbb{R}^3 \cap \{u< 2\} } u \d x\\
&\le
\int_{ \mathbb{R}^3 } u^2(u-1)^2 \d x
+2M
.\label{L2 square control}
\end{align}
%We need to control
%\[\int_{ \mathbb{R}^3 } u^2(u-1)^2 \d x,\]
%which is part of the energy $\mathcal{E}_\ZT(u)$.
Since clearly
\begin{align}
\int_{ \mathbb{R}^3 } u^2(u-1)^2 \d x +\mathcal{D}(u) & \le \underbrace{ \mathcal{E}_\ZT(u)}_{\le 0}+\ZT\int_{ B_1 } \frac{u}{|x|}\d x
+\ZT\underbrace{\int_{ \mathbb{R}^3\setminus B_1 } \frac{u}{|x|}\d x }_{\le M}\notag\\
&
\le \ZT M+\ZT\int_{ B_1 } \frac{u^2}{2\ZT}+ \frac{\ZT}{|x|^2}\d x \\
&=\ZT M+ \frac{\|u\|_{ \mathcal{L}^2(B_1)}^2}{2}+ \ZT^2\int_{ B_1 }  \frac{1}{|x|^2}\d x  \notag\\
&\le \ZT M+ \frac{\|u\|_{ \mathcal{L}^2(B_1)}^2}{2}+ 4\pi \ZT^2\\
&\le \ZT M+ \frac{\|u\|_{ \Ltwo }^2}{2}+ 4\pi \ZT^2,\label{double control}
\end{align}
which, plugged into \eqref{L2 square control}, gives
\begin{align*}
\int_{ \mathbb{R}^3 } u^2 \d x & \le
\int_{ \mathbb{R}^3 } u^2(u-1)^2 \d x
+2M
\le (\ZT+2)M+ \frac{\|u\|_{ \Ltwo }^2}{2}+ 4\pi \ZT^2,\end{align*}
so that
\begin{align*}
    \|u\|_{ \Ltwo }^2\le
2(\ZT+2)M+ 8\pi \ZT^2,
\end{align*}
hence \eqref{relating L2 and L1 norms - equation} is proven. Combining \eqref{double control} and \eqref{relating L2 and L1 norms - equation}
gives
\eqref{relating D and L2 norms - equation},
concluding the proof.
%%%%it suffices to get an upper bound for $\int_{ B_1 } \frac{u}{|x|}\d x$. Observe that
%%%%\begin{align*}
%%%%0&\ge \mathcal{E}_\ZT(u) = \int_{ \mathbb{R}^3 } \Big[\frac{|\nabla u|^2}{2}+ u^2(u-1)^2 - \frac{\ZT u}{|x|} \Big]\d x
%%%%+\mathcal{D}(u)\\
%%%%&\ge
%%%% \int_{ \mathbb{R}^3 } \Big[ u^2(u-1)^2 - \frac{\ZT u}{|x|} \Big]\d x\\
%%%% &=
%%%% \int_{ \mathbb{R}^3\setminus B_1 } \Big[ u^2(u-1)^2 - \frac{\ZT u}{|x|} \Big]\d x
%%%% +\int_{ B_1 } \Big[ u^2(u-1)^2 - \frac{\ZT u}{|x|} \Big]\d x\\
%%%% &\ge -\ZT M+\int_{ B_1 } \Big[ u^2(u-1)^2 - \frac{\ZT u}{|x|} \Big]\d x,
%%%%\end{align*}
%%%%thus, noting that
%%%%\[ u^2(u-1)^2 \ge u^4-2u^3 \ge \frac{u^4}{3} \quad\text{on } \{u>3\}, \]
%%%%we get
%%%%\begin{align*}
%%%%\ZT M &\ge \int_{ B_1  } \Big[ u^2(u-1)^2 - \frac{\ZT u}{|x|} \Big]\d x \\
%%%%&=\int_{ B_1 \cap \{u>3\} } \Big[ u^2(u-1)^2 - \frac{\ZT u}{|x|} \Big]\d x
%%%%+\int_{ B_1 \cap \{u\le 3\} } \Big[ u^2(u-1)^2 - \frac{\ZT u}{|x|} \Big]\d x\\
%%%%&\ge
%%%%\int_{ B_1 \cap \{u>3\} } \Big[ \frac{u^4}{3}- \frac{\ZT u}{|x|} \Big]\d x  - 6\pi \ZT \\
%%%%&\ge
%%%%\underbrace{\int_{ \{x\in B_1 \cap \{u>3\}: u\ge \sqrt[3]{3\ZT /|x|} \}} \Big[ \frac{u^4}{3}- \frac{\ZT u}{|x|} \Big]\d x }_{\ge 0}
%%%%-\int_{ \{x\in B_1 : 3<u< \sqrt[3]{3\ZT /|x|} \}} \frac{\ZT u}{|x|} \d x - 6\pi \ZT \\
%%%%&\ge
%%%%-\ZT \int_{  B_1 } \frac{1}{|x|}\sqrt[3]{\frac{3\ZT }{|x|}} \d x - 6\pi \ZT
%%%%=\frac{12\pi\sqrt[3]{3}}{5} \ZT ^{4/3}- 6\pi \ZT
%%%%\end{align*}
%
\end{proof}

\begin{Lmm}(Improved version of the previous Lemma)\label{relating L2 and L1 norms - improved lemma}
Assume there exists a minimizer $u\in \Honehat$ with $\|u\|_{ \Lone }=M$. Then it holds
\begin{align}\label{relating L2 and L1 norms - improved equation}\vert\vert u\vert\vert_{\Ltwo}^2&\leq C(\ZT^2+\ZT^6),\\
\label{relating D and L2 norms - improved equation}\mathcal{D}(u)&\leq C\ZT^2.\end{align}
\end{Lmm}
\begin{proof}These are an immediate consequence of the nonpositivity of $\IV(M)=\EV(u)$, and equations \eqref{Crcv1} and \eqref{Crcv2}.\end{proof}

We finalize this section by establishing estimates that play a central role in the proof of part (b) of Theorem \ref{Th1}. The main idea is to use use localization functions to extract information on how the mass of minimizers is distributed in $\RRR$.

We will use a suitably modified version of Lemmas 3.1 and 3.2 from \cite{FrankNamVanDenBosch}.

\begin{Lmm}\label{analogue of Lemma 3.1 of Frank-Nam-Van Den Bosch}
	(Analogue of \cite[Lemma~3.1]{FrankNamVanDenBosch})
	For all smooth partitions of unity $f_i:\mathbb{R}^3 \longrightarrow [0,1]$, $i=1,\cdots,n$,
	such that $\sum_{i=1}^{n}f_i^2 =1$, $\nabla f_i \in  \Linfty $, and for all
	$u :\mathbb{R}^3 \longrightarrow [0,+\infty]$ such that $u \in H^1(\mathbb{R}^3)$,
	it holds
	\begin{align*}
	\sum_{i=1}^n \mathcal{E}_\ZT (f_i^2 u) -\mathcal{E}_\ZT ( u)
	&\le
		\sum_{i=1}^n \mathcal{D}(f_i^2 u) -\mathcal{D}( u)\\
		&\qquad
		+\Big[\sum_{i=1}^n\| \nabla f_i \|_{  \Linfty }^2\Big]\int_A u^2 \d x
		+
		 \min\bigg\{  4\int_{A} u^2\d x,
		8\int_{A} u\d x \bigg\},
	\end{align*}
	where
	$A:=\bigcup_{i=1}^n \{0<f_i<1\}$.
\end{Lmm}

\begin{proof}
	The Coulomb repulsion part is exactly the same as in \cite[Lemma~3.1]{FrankNamVanDenBosch}.

	\medskip

	{\em Gradient term.} Again, as done in \cite[Lemma~3.1]{FrankNamVanDenBosch}, we apply the IMS formula
	\begin{align} \sum_{i=1}^n \int_{\mathbb{R}^3} |\nabla(f_i^2 \sqrt{\rho})|^2 \d x -
	\int_{\mathbb{R}^3} |\nabla \sqrt{\rho}|^2 \d x &=
	 \int_{\mathbb{R}^3} \Big(\sum_{i=1}^n |\nabla f_i|^2\Big) \rho \d x \\
	 &\le  \Big(\sum_{i=1}^n \|\nabla f_i\|_{  \Linfty }^2\Big)	 \int_{A} \rho \d x  \end{align}
	 with $\rho=u^2$, hence
	\[ \sum_{i=1}^n \int_{\mathbb{R}^3} |\nabla(f_i^2 u)|^2 \d x -
	\int_{\mathbb{R}^3} |\nabla u|^2 \d x
	\le  \Big(\sum_{i=1}^n \|\nabla f_i\|_{  \Linfty }^2\Big)	 \int_{A} u^2 \d x.  \]

	 \medskip

	 {\em Double well term.} Direct computations give
	 \begin{align*}
	 \sum_{i=1}^n &\int_{\mathbb{R}^3} f_i^4 u^2(1-f_i^2 u )^2 \d x -
	 \int_{\mathbb{R}^3} u^2(1-u)^2 \d x\\
	 &= \int_{\mathbb{R}^3} \sum_{i=1}^n \Big( (f_i^2 u)^4-2(f_i^2 u)^2+(f_i^2 u)^2  \Big)\d x -
	 \int_{\mathbb{R}^3} (u^4-2u^3+u^2) \d x\\
	 &= \int_{A} \Big[\sum_{i=1}^n \Big( (f_i^2 u)^4-2(f_i^2 u)^2+(f_i^2 u)^2  \Big) -
	  (u^4-2u^3+u^2)\Big]  \d x,
	 \end{align*}
since outside of $A$ we have $f_i=0$ for all but one index (that we call $j$), and condition $\sum_{i=1}^n f_i^2=1$
forces $f_j=1$. The above inequality then continues as
\begin{align*}
\int_{A} &\Big[\sum_{i=1}^n \Big( (f_i^2 u)^4-2(f_i^2 u)^2+(f_i^2 u)^2  \Big) -
(u^4-2u^3+u^2)\Big]  \d x\\
&=\int_{A} \Big[  u^4\Big(\sum_{i=1}^n f_i^8 -1\Big)
-2u^3\Big(\sum_{i=1}^n f_i^6 -1\Big)
+u^2\Big( \underbrace{ \sum_{i=1}^n f_i^4 -1}_{< 0}\Big) \Big]  \d x\\
&\le
\int_{A} \Big[ 2u^3\Big(1-\sum_{i=1}^n f_i^6 \Big)- u^4\Big(1-\sum_{i=1}^n f_i^8 \Big)
   \Big]  \d x\\
   &\le
   \int_{A} (2u^3- u^4)\Big(1-\sum_{i=1}^n f_i^6 \Big)\d x
   \\
   &\le
   \int_{A\cap \{u\le 2\}} 2u^3\d x \\
   &\le \min\bigg\{  4\int_{A\cap \{u\le 2\}} u^2\d x,
     8\int_{A\cap \{u\le 2\}} u\d x \bigg\}\\
     &\le \min\bigg\{  4\int_{A} u^2\d x,
     8\int_{A} u\d x \bigg\},
\end{align*}
and the proof is complete.
\end{proof}

\begin{Lmm}\label{analogue of Lemma 3.2 of Frank-Nam-Van Den Bosch}
	(Analogue of \cite[Equation~(22)]{FrankNamVanDenBosch})
	Assume there exists a minimizer $u\in\Honehat$ with $\|u\|_{ \Lone }=M$.
	For all $r,s>0$, $0<\lambda\le 1/2$, we have
	\begin{align}
	\frac{1}{8}\bigg(\int_{\mathbb{R}^3} \chi_{(1+\lambda)r}^+ u\d x \bigg)^2 &\le
	2s\mathcal{D}(\chi_{(1+\lambda)r}^+ u)
	+\frac{C}{\lambda^2 s^2} \int_{ \mathbb{R}^3 }\chi_{r}^+ u^2\d x \\
	&\ \ \ +\bigg( 8+\frac{1}{4}\Big[ \sup_{|z|\ge r} |z|\Phi_r(z) \Big] \bigg)\int_{ \mathbb{R}^3 }\chi_{r}^+ u\d x,
	\label{analogue of (22)}
	\end{align}
	where $\chi_{r}^+:=\mathbbm{1}_{|x|\ge r}$, and
	\[ \Phi_r(x):= \frac{\ZT }{|x|} - \int_{B_r} \frac{u(y)}{|x-y|}\d y. \]
\end{Lmm}

\begin{proof}
Assume there exists a minimizer $u\in \Honehat$. As done in \cite[Lemma~3.2]{FrankNamVanDenBosch},
	for parameters $s,\ell,\lambda>0$,
	choose partitions of unity $\chi_i$, $i=1,2$, such that
	$\chi_1^2+\chi_2^2=1$, and
	\[ \chi_i (x):=g_i\left(\frac{\nu\cdot \theta(x)-\ell}{s}\right) ,\qquad i=1,2,\]
	where $\nu$ denotes the exterior unit normal to the ball
	$\{|x|\le r\}$, $g_i:\mathbb{R}\longrightarrow \mathbb{R}$, $i=1,2$, are smooth functions such that
	\[g_1^2+g_2^2=1,\quad |g_1'|^2+|g_2'|^2\le C,\quad
	g_1(t)=1 \text{ if } t\le 0,\quad g_1(t)=0 \text{ if } t\ge 1,\]
	$s,\ell$ are parameters to be chosen later, and $\theta:\mathbb{R}^3\longrightarrow \mathbb{R}^3$ is a radial function satisfying
	\[|\theta(x)|\le |x|,\quad \theta(x)=0 \text{ if } |x|\le r,\quad
	\theta(x)=x \text{ if } |x|\ge (1+\lambda) r,\quad
	|\nabla \theta| \le \frac{C}{\lambda}.
	%\quad	|\nabla^2 \theta| \le \frac{C}{\lambda^2}.
	\]
	Consequently,
	\[ \chi_1(x) = 1 \text{ if } \nu\cdot \theta(x)\le \ell,\qquad
	\chi_1(x) = 0 \text{ if } \nu\cdot \theta(x)\ge \ell+s . \]
	By Lemma \ref{analogue of Lemma 3.1 of Frank-Nam-Van Den Bosch}, using the minimality of $u$ we have
	\begin{align}
	0& \le \mathcal{E}_\ZT (\chi_1^2 u )+\mathcal{E}_{\ZT =0}(\chi_2^2 u )-\mathcal{E}_\ZT (u )\notag\\
	&\le \ZT \int_{\mathbb{R}^3}\frac{\chi_2^2 u}{|x|}\d x +
	\mathcal{D}(\chi_1^2 u)+\mathcal{D}(\chi_2^2 u)-\mathcal{D}(u)\notag\\
	&\qquad
	+\frac{C}{\lambda^2 s^2} \int_{\nu\cdot \theta(x)-s\le \ell \le \nu\cdot \theta(x)} u^2\d x
	+8 \int_{\nu\cdot \theta(x)-s\le \ell \le \nu\cdot \theta(x)} u\d x.
	\label{binding inequality}
	\end{align}
	The attraction and Coulomb repulsion terms are estimated exactly as in \cite[Lemma~3.2]{FrankNamVanDenBosch},
	hence \eqref{binding inequality} gives
	\begin{align}
	\int&\int_{ \substack{ |x|,|y|\ge (1+\lambda)r \\ \nu\cdot y\le \ell \le \nu\cdot x-s }} \frac{u(x)u(y)}{|x-y|}\d x\d y\notag\\
	&\le
	\int_{\ell \le x\cdot \theta(x)} u(x) [\Phi_r(x)]_+\d x
	+\frac{C}{\lambda^2 s^2} \int_{\nu\cdot \theta(x)-s\le \ell \le \nu\cdot \theta(x)} u^2\d x
	+8 \int_{\nu\cdot \theta(x)-s\le \ell \le \nu\cdot \theta(x)} u\d x,
	\label{analogue of (19)}
	\end{align}
	where $[\cdot]_+$ denotes the positive part.
	By arguing like in \cite[Lemma~3.2]{FrankNamVanDenBosch},
	 we can get
	\begin{align*}
	\frac{1}{8}\bigg(\int_{\mathbb{R}^3} \chi_{(1+\lambda)r}^+ u\d x \bigg)^2 &\le
	2s\mathcal{D}(\chi_{(1+\lambda)r}^+ u)
	+\frac{C}{\lambda^2 s^2} \int_{ \mathbb{R}^3 }\chi_{r}^+  u^2\d x\\
	&\ \ \ +\bigg( 8+\frac{1}{4}\Big[ \sup_{|z|\ge r} |z|\Phi_r(z) \Big] \bigg)\int_{ \mathbb{R}^3 }\chi_{r}^+ u\d x,
	\end{align*}
		which is the analogue of \cite[Equation~(22)]{FrankNamVanDenBosch}, and the proof is complete.
\end{proof}

The key difference between our Lemma \ref{analogue of Lemma 3.2 of Frank-Nam-Van Den Bosch} and
\cite[Lemma~3.2]{FrankNamVanDenBosch} is that we have
\[\int_{ \mathbb{R}^3 }\chi_{r}^+  u^2\d x\]
in the upper bound on the right hand side of \eqref{analogue of (22)},
instead of
\[\int_{ \mathbb{R}^3 }\chi_{r}^+  u\d x\]
as in
\cite[Equation~(22)]{FrankNamVanDenBosch}. Therefore, we cannot apply the arguments from \cite[Lemma~3.3]{FrankNamVanDenBosch},
since
\[\lim_{r\to 0}\int_{ \mathbb{R}^3 }\chi_{r}^+  u^2\d x = \|u\|_{ \Ltwo }^2\]
might be different from
\[\lim_{r\to 0}\int_{ \mathbb{R}^3 }\chi_{r}^+  u\d x = \|u\|_{ \Lone }=M.\]
We use Lemma \ref{relating L2 and L1 norms - improved lemma} to relate $\|u\|_{ \Ltwo }^2$ and $\|u\|_{ \Lone }=M$.

%Proof of Thm about minimizers
\section{Proof of part (a) of Theorem \ref{Th1}}

\begin{proof}[Proof of part (a) of Theorem \ref{Th1}]By part Lemma \ref{PrprtsIOK}, there exists $0<m\leq M\leq\ZT$ such that $\IV(M)=\IV(m)$ and $\IV(m)$ is attained, say at $u$ with $\vert\vert u\vert\vert_{\Lone}=m>0$. If $m<M$, then $\IV(M)=\IV(\alpha)$ for $m\leq \alpha\leq M$ by the nonincreasingness of $\IV$. Then $u$ satisfies the hypothesis of Proposition \ref{NoSlnsZrEig}, so that $m\geq\ZT$. However, this contradicts $m<M\leq\ZT$. Therefore, $m=M$ and $\IV(M)$ is attained.
\end{proof}

\section{Proof of part (b) of Theorem \ref{Th1}}

\begin{proof}[Proof of part (b) of Theorem	\ref{Th1} using Lemma \ref{relating L2 and L1 norms - lemma}]

%Now we are rea\d y to prove the non existence result stated in Theorem \ref{non existence for large masses}.
Assume there exists a minimizer $u\in H^1(\mathbb{\R}^3)$ with $\|u\|_{ \Lone }=M$.
Taking the limit
$r\to 0^+$ in \eqref{analogue of (22)} and applying \eqref{relating L2 and L1 norms - equation} and \eqref{relating D and L2 norms - equation} gives
\begin{align}
\frac{1}{8}M^2 &\le
2s\mathcal{D}( u)
+\frac{C}{\lambda^2 s^2} \int_{ \mathbb{R}^3 } u^2\d x
+\bigg( 8+\frac{1}{4}\Big[ \sup_{|z|\ge 0} |z|\Phi_r(z) \Big] \bigg)M \notag\\
&
{\le}
2[2(\ZT+1)M+ 8\pi \ZT^2]s
+\frac{C[2(\ZT+2)M+ 8\pi \ZT^2]}{\lambda^2 s^2}
+\bigg( 8+\frac{\ZT}{4} \bigg)M,
\label{M not too large}
\end{align}
which must hold for all $\lambda\in (0,1/2]$, and $s>0$. We can
choose $\lambda=1/2$ and optimize over $s>0$, and note that the left hand side $M^2/8$
grows like $O(M^2)$, while the upper bound in the right hand side of \eqref{M not too large}
grows like $O(M)$. Thus \eqref{M not too large} can hold only for $M$ not too large, and the proof is complete.
\end{proof}

\begin{proof}[Proof of part (b) of Theorem \ref{Th1} using Lemma \ref{relating L2 and L1 norms - improved lemma})]
				The only difference is that \eqref{M not too large} is replaced by
\begin{align}
\frac{1}{8}M^2 &\le
+\frac{C}{\lambda^2 s^2} \int_{ \mathbb{R}^3 } u^2\d x
+\bigg( 8+\frac{1}{4}\Big[ \sup_{|z|\ge 0} |z|\Phi_r(z) \Big] \bigg)M \notag\\
&
\overset{\eqref{relating L2 and L1 norms - improved equation},\eqref{relating D and L2 norms - improved equation}}{\le}
2Cs\ZT^2 +\frac{C (\ZT^2+\ZT^6)}{\lambda^2 s^2}
+\bigg( 8+\frac{\ZT}{4} \bigg)M ,\label{M not too large - improved}
\end{align}
		and again the left hand side term
		grows like $O(M^2)$, while the upper bound in the right hand side of \eqref{M not too large - improved}
		grows like $O(M)$. Thus \eqref{M not too large - improved} can hold only for $M$ not too large, and the proof is complete.
\end{proof}

%%%%%%%%%%%%%%%%%
%% APPENDIX
%\appendix
%\input{appendix}

%%%%%%%%%%%%%%%%%
%% REFERENCES
\bibliographystyle{siam}
\bibliography{bibitemsLatest}

\end{document}